\documentclass[12pt,a4paper]{amsart}

\usepackage{amsmath,amssymb,amsfonts,amsthm,latexsym,graphicx,multirow}%,mathdots
\usepackage{hyperref}

\usepackage{color}

\newcommand{\A}{\mathrm{A}}
\newcommand{\Alt}{\mathrm{Alt}}
\newcommand{\Sy}{\mathrm{S}}
\newcommand{\Sym}{\mathrm{Sym}}
\newcommand{\Soc}{\mathrm{Soc}}
\newcommand{\Aut}{\mathrm{Aut}}
\newcommand{\Out}{\mathrm{Out}}
\newcommand{\Hol}{\mathrm{Hol}}
\newcommand{\Cos}{\mathrm{Cos}}

\newcommand{\Inn}{\mathrm{Inn}}

\newtheorem{theorem}{Theorem}[section]
\newtheorem{lemma}[theorem]{Lemma}
\newtheorem{proposition}[theorem]{Proposition}
\newtheorem{corollary}[theorem]{Corollary}

\theoremstyle{definition}

\newtheorem{notation}[theorem]{Notation}
\newtheorem{construction}[theorem]{Construction}

\newtheorem{remark}[theorem]{Remark}
\newtheorem{question}[theorem]{Question}

\begin{document}

\title{Vertex-quasiprimitive $2$-arc-transitive digraphs}

\author[Giudici]{Michael Giudici}
\address{School of Mathematics and Statistics\\University of Western Australia\\ Crawley 6009, WA\\ Australia}
\email{michael.giudici@uwa.edu.au}

\author[Xia]{Binzhou Xia}
\address{School of Mathematics and Statistics\\University of Western Australia\\ Crawley 6009, WA\\ Australia}
\email{binzhou.xia@uwa.edu.au}

\maketitle

\begin{abstract}
We study vertex-quasiprimitive $2$-arc-transitive digraphs, reducing the problem of vertex-primitive $2$-arc-transitive digraphs to almost simple groups. This includes a complete classification of vertex-quasiprimitive $2$-arc-transitive digraphs where the action on vertices has O'Nan-Scott type SD or CD.

\textit{Key words:} digraphs; vertex-quasiprimitive; $2$-arc-transitive

\textit{MSC2010:} 05C20, 05C25
\end{abstract}

\section{Introduction}

A \emph{digraph} $\Gamma$ is a pair $(V,\rightarrow)$ with a set $V$ (of vertices) and an antisymmetric irreflexive binary relation $\rightarrow$ on $V$. All digraphs considered in this paper will be finite. For a non-negative integer $s$, an \emph{$s$-arc} of $\Gamma$ is a sequence $v_0,v_1,\dots,v_s$ of vertices with $v_i\rightarrow v_{i+1}$ for each $i=0,\dots,s-1$. A $1$-arc is also simply called an \emph{arc}. We say $\Gamma$ is \emph{$s$-arc-transitive} if the group of all automorphisms of $\Gamma$ (that is, all permutations of $V$ that preserve the relation $\rightarrow$) acts transitively on the set of $s$-arcs. More generally, for a group $G$ of automorphisms of $\Gamma$, we say $\Gamma$ is \emph{$(G,s)$-arc-transitive} if $G$ acts transitively on the set of $s$-arcs of $\Gamma$.

A transitive permutation group $G$ on a set $\Omega$ is said to be \emph{primitive} if $G$ does not preserve any nontrivial partition of $\Omega$, and is said to be \emph{quasiprimitive} if each nontrivial normal subgroup of $G$ is transitive. It is easy to see that primitive permutation groups are necessarily quasiprimitive, but there are quasiprimitive permutation groups that are not primitive. We say a digraph is \emph{vertex-primitive} if its automorphism group is primitive on the vertex set. The aim of this paper is to investigate finite vertex-primitive $s$-arc transitive digraphs with $s\geqslant2$. However, we will often work in the more general quasiprimitive setting.

There are many $s$-arc-transitive digraphs, see for example~\cite{CLP1995,MS2001,MS2004,Praeger1989}. In particular, for every integer $k\geqslant2$ and every integer $s\geqslant1$ there are infinitely many $k$-regular $(G,s)$-arc-transitive digraphs with $G$ quasiprimitive on the vertex set (see the proof of Theorem~1 of \cite{CLP1995}). On the other hand, the first known family of vertex-primitive $2$-arc-transitive digraphs besides directed cycles was only recently discovered in \cite{GLX}. The digraphs in this family are not $3$-arc-transitive, which prompted the following question:

\begin{question}\label{qes1}
Is there an upper bound on $s$ for vertex-primitive $s$-arc-transitive digraphs that are not directed cycles?
\end{question}

The O'Nan-Scott Theorem divides the finite primitive groups into eight types and there is a similar theorem for finite quasiprimitive groups, see \cite[Section~5]{Praeger1997}). For four of the eight types, a quasiprimitive group of that type has a normal regular subgroup. Praeger \cite[Theorem~3.1]{Praeger1989} showed that if $\Gamma$ is a $(G,2)$-arc-transitive digraph and $G$ has a normal subgroup that acts regularly on $V$, then $\Gamma$ is a directed cycle. Thus to investigate vertex-primitive  and vertex-quasiprimitive 2-arc-transitive digraphs, we only need to consider the four remaining types. One of these types is where $G$ is an almost simple group, that is, where $G$ has a unique minimal normal subgroup $T$, and $T$ is a nonabelian simple group. The examples of primitive 2-arc-transitive digraphs constructed in \cite{GLX} are of this type. This paper examines the remaining three types, which are labelled SD, CD and PA, and reduces Question~\ref{qes1} to almost simple vertex-primitive groups (Corollary~\ref{cor1}). We now define these three types and state our results.

We say that a quasiprimitive group $G$ on a set $\Omega$ is of type $\mathrm{SD}$ if $G$ has a unique minimal normal subgroup $N$, there exists a nonabelian simple group $T$ and positive integer $k\geqslant2$ such that $N\cong T^k$, and for $\omega\in\Omega$, $N_\omega$ is a full diagonal subgroup of $N$ (that is, $N_\omega\cong T$ and projects onto $T$ in each of the $k$ simple direct factors of $N$). It is incorrectly claimed in~\cite[Lemma~4.1]{Praeger1989} that there is no $2$-arc-transitive digraph with a vertex-primitive group of automorphisms of type $\mathrm{SD}$. However, there is an error in the proof which occurs when concluding ``$\sigma x$ also fixes $D\mathbf{t}^{-1}$". Indeed, given a nonabelian simple group $T$, our Construction~\ref{Construction} yields a $(G,2)$-arc-transitive digraph $\Gamma(T)$ with $G$ primitive of type $\mathrm{SD}$. These turn out to be the only examples.

\begin{theorem}\label{thm1}
Let $\Gamma$ be a connected $(G,2)$-arc-transitive digraph such that $G$ acts quasiprimitively of type $\mathrm{SD}$ on the set of vertices. Then there exists a nonabelian simple group $T$ such that $\Gamma\cong\Gamma(T)$, as obtained from \emph{Construction~\ref{Construction}}. Moreover, $\Aut(\Gamma)$ is vertex-primitive of type SD and $\Gamma$ is not $3$-arc-transitive.
\end{theorem}

The remaining two quasiprimitive types, $\mathrm{CD}$ and $\mathrm{PA}$, both arise from product actions. For any digraph $\Sigma$ and positive integer $m$, $\Sigma^m$ denotes the direct product of $m$ copies of $\Sigma$ as in Notation~\ref{DirectProduct}.
The wreath product $\Sym(\Delta)\wr\Sy_m=\Sym(\Delta)^m\rtimes\Sy_m$ acts naturally on the set $\Delta^m$ with product action. Let $G_1$ be the stabiliser in $G$ of the first coordinate and let $H$ be the projection of $G_1$ onto $\Sym(\Delta)$. If $G$ projects onto a transitive subgroup of $\Sy_m$, then a result of Kov\'acs~\cite[(2.2)]{kovacs} asserts that up to conjugacy in $\Sym(\Delta)^m$ we may assume that $G\leqslant H\wr\Sy_m$. A reduction for 2-arc-transitive digraphs was sought in \cite[Remark 4.3]{Praeger1989} but only partial results were obtained. Our next result yields the desired reduction.

\begin{theorem}\label{thm2}
Let $H\leqslant\Sym(\Delta)$ with transitive normal subgroup $N$ and let $G\leqslant H\wr\Sy_m$ acting on $V=\Delta^m$ with product action such that $G$  projects to a transitive subgroup of $\Sy_m$ and $G$ has component $H$. Moreover, assume that $N^m\leqslant G$. If $\Gamma$ is a $(G,s)$-arc-transitive digraph with vertex set $V$ such that $s\geqslant2$, then $\Gamma\cong\Sigma^m$ for some $(H,s)$-arc-transitive digraph $\Sigma$ with vertex set $\Delta$.
\end{theorem}

A quasiprimitive group of type $\mathrm{CD}$ on a set $\Omega$ is one that has a product action on $\Omega$ and the component is quasiprimitive of type $\mathrm{SD}$, while a quasiprimitive group of type $\mathrm{PA}$ on a set $\Omega$ is one that acts faithfully on some partition $\mathcal{P}$ of $\Omega$ and $G$ has a product action on $\mathcal{P}$ such that the component $H$ is an almost simple group. When $G$ is primitive of type $\mathrm{PA}$, $H$ is primitive and the partition $\mathcal{P}$ is the partition into singletons, that is, $G$ has a product action on $\Omega$. As a consequence, we have the following corollaries.

\begin{corollary}\label{cor2}
Suppose $\Gamma$ is a connected $(G,2)$-arc-transitive digraph such that $G$ is vertex-quasiprimitive of type $\mathrm{CD}$. Then there exists a nonabelian simple group $T$ and positive integer $m\geqslant2$ such that $\Gamma\cong\Gamma(T)^m$, where $\Gamma(T)$ is as obtained from \emph{Construction~\ref{Construction}}. Moreover, $\Gamma$ is not $3$-arc-transitive.
\end{corollary}

\begin{corollary}\label{cor3}
Suppose $\Gamma$ is a $(G,s)$-arc-transitive digraph such that $G$ is vertex-primitive of type $\mathrm{PA}$. Then $\Gamma\cong\Sigma^m$ for some $(H,s)$-arc-transitive digraph $\Sigma$ and integer $m\geqslant2$ for some almost simple primitive permutation group $H\leqslant\Aut(\Sigma)$.
\end{corollary}

We give an example in Section~\ref{sec2} of an infinite family of $(G,2)$-arc-transitive digraphs $\Gamma$ with $G$ vertex-quasiprimitive of PA type such that $\Gamma$ is not a direct power of a digraph $\Sigma$ (indeed the number of vertices of $\Gamma$ is not a proper power). We leave the investigation of such digraphs open.

We note that Theorem \ref{thm1} and Corollaries \ref{cor2} and \ref{cor3}, reduce Question~\ref{qes1} to studying almost simple primitive groups.

\begin{corollary}\label{cor1}
There exists an absolute upper bound $C$ such that every vertex-primitive $s$-arc-transitive digraph that is not a directed cycle satisfies $s\leqslant C$, if and only if for every $(G,s)$-arc-transitive digraph with $G$ a primitive almost simple group we have $s\leqslant C$.
\end{corollary}

Theorem~\ref{thm1} follows immediately from a more general theorem (Theorem~\ref{thm3}) given at the end of Section~3. Then in Section~4, we prove Theorem~\ref{thm2} as well as Corollaries~\ref{cor2}--\ref{cor3} after establishing some general results for normal subgroups of $s$-arc-transitive groups.

\section{Preliminaries}\label{sec1}

We say that a digraph $\Gamma$ is \emph{$k$-regular} if both the set $\Gamma^-(v)=\{u\in V\mid u\rightarrow v\}$ of in-neighbours of $v$ and the set $\Gamma^+(v)=\{w\in V\mid v\rightarrow w\}$ of out-neighbours of $v$ have size $k$ for all $v\in V$, and we say that $\Gamma$ is  \emph{regular} if it is $k$-regular for some positive integer $k$. Note that any vertex-transitive digraph is regular. Moreover, if $\Gamma$ is regular and $(G,s)$-arc-transitive with $s\geqslant2$ then it is also $(G,s-1)$-arc-transitive.

Recall that a digraph is said to be connected if and only if its underlying graph is connected. A vertex-primitive digraph is necessarily connected, for otherwise its connected components would form a partition of the vertex set that is invariant under digraph automorphisms.

\subsection{Group factorizations}

All the groups we consider in this paper are assumed to be finite. An expression of a group $G$ as the product of two subgroups $H$ and $K$ of $G$ is called a \emph{factorization} of $G$. The following lemma lists several equivalent conditions for a group factorization, whose proof is fairly easy and so is omitted.

\begin{lemma}\label{Factorization}
Let $H$ and $K$ be subgroups of $G$. Then the following are equivalent.
\begin{itemize}
\item[(a)] $G=HK$.
\item[(b)] $G=KH$.
\item[(c)] $G=(x^{-1}Hx)(y^{-1}Ky)$ for any $x,y\in G$.
\item[(d)] $|H\cap K||G|=|H||K|$.
\item[(e)] $H$ acts transitively on the set of right cosets of $K$ in $G$ by right multiplication.
\item[(f)] $K$ acts transitively on the set of right cosets of $H$ in $G$ by right multiplication.
\end{itemize}
\end{lemma}

The $s$-arc-transitivity of digraphs can be characterized by group factorizations as follows:

\begin{lemma}\label{lem19}
Let $\Gamma$ be a $G$-arc-transitive digraph, $s\geqslant2$ be an integer, and $v_0\rightarrow v_1\rightarrow\dots\rightarrow v_{s-1}\rightarrow v_s$ be an $s$-arc of $\Gamma$. Then $\Gamma$ is $(G,s)$-arc-transitive if and only if $G_{v_1\dots v_i}=G_{v_0v_1\dots v_i}G_{v_1\dots v_iv_{i+1}}$ for each $i$ in $\{1,\dots,s-1\}$.
\end{lemma}

\begin{proof}
For any $i$ such that $1\leqslant i\leqslant s-1$, the group $G_{v_1\dots v_i}$ acts on the set $\Gamma^+(v_i)$ of out-neighbours of $v_i$. Since $v_{i+1}\in\Gamma^+(v_i)$ and $G_{v_1\dots v_iv_{i+1}}$ is the stabilizer in $G_{v_1\dots v_i}$ of $v_{i+1}$, by Frattini's argument, the subgroup $G_{v_0v_1\dots v_i}$ of $G_{v_1\dots v_i}$ is transitive on $\Gamma_+(v_i)$ if and only if $G_{v_1\dots v_i}=G_{v_0v_1\dots v_i}G_{v_1\dots v_iv_{i+1}}$. Note that $\Gamma$ is $(G,s)$-arc-transitive if and only if $\Gamma$ is $(G,s-1)$-arc-transitive and $G_{v_0v_1\dots v_i}$ is transitive on $\Gamma_+(v_i)$. One then deduces inductively that $\Gamma$ is $(G,s)$-arc-transitive if and only if $G_{v_1\dots v_i}=G_{v_0v_1\dots v_i}G_{v_1\dots v_iv_{i+1}}$ for each $i$ in $\{1,\dots,s-1\}$.
\end{proof}

If $\Gamma$ is a $G$-arc-transitive digraph and $u\rightarrow v$ is an arc of $\Gamma$, then since $G$ is vertex-transitive we can write $v=u^g$ for some $g\in G$ and it follows that
\begin{equation}\label{eq9}
v^{g^{-1}}\rightarrow v\rightarrow\dots\rightarrow v^{g^{s-2}}\rightarrow v^{g^{s-1}}
\end{equation}
is an $s$-arc of $\Gamma$. In this setting, Lemma~\ref{lem19} is reformulated as follows.

\begin{lemma}\label{lem16}
Let $\Gamma$ be a $G$-arc-transitive digraph, $s\geqslant2$ be an integer, $v$ be a vertex of $\Gamma$, and $g\in G$ such that $v\rightarrow v^g$. Then $\Gamma$ is $(G,s)$-arc-transitive if and only if
$$
\bigcap_{j=0}^{i-1}g^{-j}G_vg^j=\left(\bigcap_{j=0}^ig^{-(j-1)}G_vg^{j-1}\right)\left(\bigcap_{j=0}^ig^{-j}G_vg^j\right)
$$
for each $i$ in $\{1,\dots,s-1\}$.
\end{lemma}

\begin{proof}
Let $v_j=v^{g^{j-1}}$ for any integer $j$ such that $0\leqslant j\leqslant s-1$. Then the $s$-arc~\eqref{eq9} of $\Gamma$ turns out to be $v_0\rightarrow v_1\rightarrow\dots\rightarrow v_{s-1}\rightarrow v_s$, and for any $i$ in $\{1,\dots,s\}$ we have
$$
G_{v_1\dots v_i}=\bigcap_{j=1}^iG_{v_j}=\bigcap_{j=1}^ig^{-(j-1)}G_vg^{j-1}=\bigcap_{j=0}^{i-1}g^{-j}G_vg^j
$$
and
$$
G_{v_0v_1\dots v_i}=\bigcap_{j=0}^iG_{v_j}=\bigcap_{j=0}^ig^{-(j-1)}G_vg^{j-1}.
$$
Hence the conclusion of the lemma follows from Lemma~\ref{lem19}.
\end{proof}

\subsection{Constructions of $s$-arc-transitive digraphs}

Let $G$ be a group, $H$ be a subgroup of $G$, $V$ be the set of right cosets of $H$ in $G$ and $g$ be an element of $G\setminus H$ such that $g^{-1}\notin HgH$. Define a binary relation $\rightarrow$ on $V$ by letting $Hx\rightarrow Hy$ if and only if $yx^{-1}\in HgH$ for any $x,y\in G$. Then $(V,\rightarrow)$ is a digraph, denoted by $\Cos(G,H,g)$. Right multiplication gives an action $R_H$ of $G$ on $V$ that preserves the relation $\rightarrow$, so that $R_H(G)$ is a group of automorphisms of $\Cos(G,H,g)$.

\begin{lemma}\label{CosetDigraph}
In the above notation, the following hold.
\begin{itemize}
\item[(a)] $\Cos(G,H,g)$ is $|H{:}H\cap g^{-1}Hg|$-regular.
\item[(b)] $\Cos(G,H,g)$ is $R_H(G)$-arc-transitive.
\item[(c)] $\Cos(G,H,g)$ is connected if and only if $\langle H,g\rangle=G$.
\item[(d)] $\Cos(G,H,g)$ is $R_H(G)$-vertex-primitive if and only if $H$ is maximal in $G$.
\item[(e)] Let $s\geqslant2$ be an integer. Then $\Cos(G,H,g)$ is $(R_H(G),s)$-arc-transitive if and only if for each $i$ in $\{1,\dots,s-1\}$,
$$
\bigcap_{j=0}^{i-1}g^{-j}Hg^j=\left(\bigcap_{j=0}^ig^{-(j-1)}Hg^{j-1}\right)\left(\bigcap_{j=0}^ig^{-j}Hg^j\right).
$$
\end{itemize}
\end{lemma}

\begin{proof}
Parts~(a)--(d) are folklore (see for example~\cite{CLP1995}), and part~(e) is derived in light of Lemma~\ref{lem16}.
\end{proof}

\begin{remark}\label{rem1}
Lemma~\ref{CosetDigraph} establishes a group theoretic approach to constructing $s$-arc-transitive digraphs. In particular, $\Cos(G,H,g)$ is $(R_H(G),2)$-arc-transitive if and only if $H=(gHg^{-1}\cap H)(H\cap g^{-1}Hg)$.
\end{remark}

Next we show how to construct $s$-arc-transitive digraphs from existing ones. Let $\Gamma$ be a digraph with vertex set $U$ and $\Sigma$ be a digraph with vertex set $V$. The \emph{direct product} of $\Gamma$ and $\Sigma$, denoted $\Gamma\times\Sigma$, is the digraph (it is easy to verify that this is indeed a digraph) with vertex set $U\times V$ and $(u_1,v_1)\rightarrow(u_2,v_2)$ if and only if $u_1\rightarrow u_2$ and $v_1\rightarrow v_2$, where $u_i\in U$ and $v_i\in V$ for $i=1,2$.

\begin{notation}\label{DirectProduct}
For any digraph $\Sigma$ and positive integer $m$, denote by $\Sigma^m$ the direct product of $m$ copies of $\Sigma$.
\end{notation}

%The following lemma is also stated in~\cite[Proposition~4.2]{Praeger1989}.

\begin{lemma}\label{lem15}
Let $s$ be a positive integer, $\Gamma$ be a $(G,s)$-arc-transitive digraph and $\Sigma$ be a $(H,s)$-arc-transitive digraph. Then $\Gamma\times\Sigma$ is a $(G\times H,s)$-arc-transitive digraph, where $G\times H$ acts on the vertex set of $\Gamma\times\Sigma$ by product action.
\end{lemma}

\begin{proof}
Let $(u_0,v_0)\rightarrow(u_1,v_1)\rightarrow\dots\rightarrow(u_s,v_s)$ and $(u'_0,v'_0)\rightarrow(u'_1,v'_1)\rightarrow\dots\rightarrow(u'_s,v'_s)$ be any two $s$-arcs of $\Gamma\times\Sigma$. Then $u_0\rightarrow u_1\rightarrow\dots\rightarrow u_s$ and $u'_0\rightarrow u'_1\rightarrow\dots\rightarrow u'_s$ are $s$-arcs of $\Gamma$ while $v_0\rightarrow v_1\rightarrow\dots\rightarrow v_s$ and $v'_0\rightarrow v'_1\rightarrow\dots\rightarrow v'_s$ are $s$-arcs of $\Sigma$. Since $\Gamma$ is $(G,s)$-arc-transitive, there exists $g\in G$ such that $u_i^g=u'_i$ for each $i$ with $0\leqslant i\leqslant s$. Similarly, there exists $h\in H$ such that $v_i^h=v'_i$ for each $i$ with $0\leqslant i\leqslant s$. It follows that $(u_i,v_i)^{(g,h)}=(u'_i,v'_i)$ for each $i$ with $0\leqslant i\leqslant s$. This means that $\Gamma\times\Sigma$ is a $(G\times H,s)$-arc-transitive.
\end{proof}

\subsection{Example}\label{sec2}

In this subsection we give an example of an infinite family of $(G,2)$-arc-transitive digraphs $\Gamma$ with $G$ vertex-quasiprimitive of PA type such that $\Gamma$ is not a direct power of a digraph $\Sigma$. In fact, we prove in Lemma~\ref{lem20} that the number of vertices of $\Gamma$ is not a proper power.

Let $n\geqslant5$ be odd, $G_1=\Alt(\{1,2,\dots,n\})$ and $G_2=\Alt(\{n+1,n+2,\dots,2n\})$. Take permutations
$$
a=(1,n+1)(2,n+2)\cdots(n,2n),\quad b=(1,2)(3,4)(n+1,n+2)(n+3,n+4)
$$
and
\begin{align*}
g=(&1,n+2,2,n+3,5,n+6,7,n+8,\dots,2i-1,n+2i,\dots,n-2,2n-1,n,\\
&n+1,3,n+4,4,n+5,6,n+7,\dots,2j,n+2j+1,\dots,n-1,2n).
\end{align*}
In fact, $g=ac$ with
\[
c=(1,3,5,6,7,\dots,n)(n+1,n+2,\dots,2n).
\]
Let $G=(G_1\times G_2)\rtimes\langle a\rangle$, and note that $g\in G$ as $c\in G_1\times G_2$. Let $H=\langle a,b\rangle=\langle a\rangle\times\langle b\rangle$ and $\Gamma_n=\Cos(G,H,g)$.

\begin{lemma}
For all odd $n\geqslant5$, $\Gamma_n$ is a connected $(G,2)$-arc-transitive digraph with $G$ quasiprimitive of PA type on the vertex set.
\end{lemma}

\begin{proof}
As $(G_1\times G_2)\cap H=\langle b\rangle$ we see that $G$ is quasiprimitive of PA type on the vertex set. To show that $\Gamma_n$ is connected, we shall show $\langle H,g\rangle=G$ in light of Lemma~\ref{CosetDigraph}(c). Let $M=\langle H,g\rangle\cap(G_1\times G_2)$. Then we only need to show $M=G_1\times G_2$ since $a\in\langle H,g\rangle$.

Denote the projections of $G_1\times G_2$ onto $G_1$ and $G_2$, respectively, by $\pi_1$ and $\pi_2$. Note that $g^2$ fixes both $\{1,\dots,n\}$ and $\{n+1,\dots,2n\}$ setwise with
%\begin{align*}
%g^2=(&1,2,5,7,\dots,2i-1,\dots,n,3,4,6,\dots,2j,\dots,n-1)(n+1,n+4,n+5,n+7,\\
%&\dots,n+2j+1,\dots,2n,n+2,n+3,n+6,n+8,\dots,n+2i,\dots,2n-1)
%\end{align*}
\[
\pi_1(g^2)=(1,2,5,7,\dots,2i-1,\dots,n,3,4,6,\dots,2j,\dots,n-1)
\]
and
\[
\pi_1(g^{n+1})=(1,3,2,4,5,\dots,n).
\]
We have $g^2\in M$ and
\[
\pi_1(g^{-(n+1)}bg^{n+1}b)=\pi_1(g^{-(n+1)}bg^{n+1})\pi_1(b)=(3,4)(2,5)(1,2)(3,4)=(1,2,5),
\]
which implies
\[
\pi_1(M)\geqslant\pi_1(\langle g^2,b\rangle)\geqslant\pi_1(\langle g^2,g^{-(n+1)}bg^{n+1}b\rangle)
=\langle\pi_1(g^2),\pi_1(g^{-(n+1)}bg^{n+1}b)\rangle=G_1
\]
using the fact that the permutation group generated by a $3$-cycle $(\alpha,\beta,\gamma)$ and an $n$-cycle with first $3$-entries $\alpha,\beta,\gamma$ is $\A_n$. It follows that
\[
\pi_2(M)=\pi_2(M^a)=(\pi_2(M))^a=G_1^a=G_2,
\]
and so $M$ is either $G_1\times G_2$ or a full diagonal subgroup of $G_1\times G_2$. However, $c=ag\in M$ while $\pi_1(c)$ and $\pi_2(c)$ have different cycle types. We conclude that $M$ is not a diagonal subgroup of $G_1\times G_2$, and so $M=G_1\times G_2$ as desired.

Now we prove that $\Gamma_n$ is $(G,2)$-arc-transitive, which is equivalent to proving that $H=(gHg^{-1}\cap H)(H\cap g^{-1}Hg)$ according to Lemma~\ref{CosetDigraph}(e). In view of
\begin{equation}\label{eq4}
(ab)^g=(ab)^{ac}=(ab)^c=a
\end{equation}
we deduce that $a\in H\cap H^g$. Since $H$ is not normal in $G=\langle H,g\rangle$, we have $H^g\neq H$. Consequently, $H\cap H^g=\langle a\rangle$. Then again by~\eqref{eq4} we deduce that
\[
H\cap H^{g^{-1}}=(H\cap H^g)^{g^{-1}}=\langle a\rangle^{g^{-1}}=\langle a^{g^{-1}}\rangle=\langle ab\rangle.
\]
This yields
\begin{equation}\label{eq11}
(gHg^{-1}\cap H)(H\cap g^{-1}Hg)=\langle a\rangle\langle ab\rangle=H.
\end{equation}

Finally, the condition $g^{-1}\notin HgH$ holds as a consequence (see~\cite[Lemma~2.3]{GLX}) of~\eqref{eq11} and the conclusion $H^g\neq H$. This completes the proof.
\end{proof}

\begin{lemma}\label{lem20}
The number of vertices of $\Gamma_n$ is not a proper power for any odd $n\geqslant5$.
\end{lemma}

\begin{proof}
Suppose that the number of vertices of $\Gamma_n$ is $m^k$ for some $m\geqslant2$ and $k\geqslant2$. Then we have
\begin{equation}\label{eq12}
m^k=\frac{|G|}{|H|}=\frac{2(n!/2)^2}{4}=\frac{(n!)^2}{8}
\end{equation}
If $k=2$, then~\eqref{eq12} gives $(n!)^2=2(2m)^2$, which is not possible. Hence $k\geqslant3$. By Bertrand's Postulate, there exists a prime number $p$ such that $n/2<p<n$. Thus, the largest $p$-power dividing $n!$ is $p$, and so the largest $p$-power dividing the right hand side of~\eqref{eq12} is $p^2$. However, this implies that the largest $p$-power dividing $m^k$ is $p^2$, contradicting the conclusion $k\geqslant3$.
\end{proof}

\subsection{Normal subgroups}

\begin{lemma}\label{lem18}
Let $\Gamma$ be a $(G,s)$-arc-transitive digraph with $s\geqslant2$, $M$ be a vertex-transitive normal subgroup of $G$, and $v_1\rightarrow\dots\rightarrow v_s$ be an $(s-1)$-arc of $\Gamma$. Then $G=MG_{v_1\dots v_i}$ for each $i$ in $\{1,\dots,s\}$.
\end{lemma}

\begin{proof}
Since $M$ is transitive on the vertex set of $\Gamma$, there exists $m\in M$ such that $v_1^m=v_2$. Denote $u_i=v_1^{m^{i-1}}$ for each $i$ such that $0\leqslant i\leqslant s$. Then $G_{u_0u_1\dots u_i}=mG_{u_1\dots u_iu_{i+1}}m^{-1}$ for each $i$ such that $0\leqslant i\leqslant s-1$, and $u_0\rightarrow u_1\rightarrow\dots\rightarrow u_s$ is an $s$-arc of $\Gamma$ since $v_1\rightarrow v_2$ and $m$ is an automorphism of $\Gamma$. For each $i$ in $\{1,\dots,s-1\}$, we deduce from Lemma~\ref{lem19} that
$$
G_{u_1\dots u_i}=G_{u_0u_1\dots u_i}G_{u_1\dots u_iu_{i+1}}=(mG_{u_1\dots u_iu_{i+1}}m^{-1})G_{u_1\dots u_iu_{i+1}}.
$$
Let $\varphi$ be the projection from $G$ to $G/M$. It follows that
\begin{eqnarray*}
\varphi(G_{u_1\dots u_i})&=&\varphi(m)\varphi(G_{u_1\dots u_iu_{i+1}})\varphi(m)^{-1}\varphi(G_{u_1\dots u_iu_{i+1}})\\
&=&\varphi(G_{u_1\dots u_iu_{i+1}})\varphi(G_{u_1\dots u_iu_{i+1}})\\
&=&\varphi(G_{u_1\dots u_iu_{i+1}})
\end{eqnarray*}
and so $G_{u_1\dots u_i}M=G_{u_1\dots u_iu_{i+1}}M$ for each $i$ in $\{1,\dots,s-1\}$. Again as $M$ is transitive on the vertex set of $\Gamma$, we have $G=MG_{u_1}$. Hence
$$
G=MG_{u_1}=MG_{u_1u_2}=\dots=MG_{u_1\dots u_i}=\dots=MG_{u_1\dots u_s}.
$$
Now for each $i$ in $\{1,\dots,s\}$, the digraph $\Gamma$ is $(G,i-1)$-arc-transitive, so there exists $g\in G$ such that $(v_1^g,\dots,v_i^g)=(u_1,\dots,u_i)$. Hence
$$
G=MG_{u_1\dots u_i}=M(g^{-1}G_{v_1\dots v_i}g)=MG_{v_1\dots v_i}
$$
by Lemma~\ref{Factorization}(c).
\end{proof}

By Frattini's argument, we have the following consequence of Lemma~\ref{lem18}:

\begin{corollary}\label{lem17}
Let $\Gamma$ be a $(G,s)$-arc-transitive digraph with $s\geqslant2$, and $M$ be a vertex-transitive normal subgroup of $G$. Then $\Gamma$ is $(M,s-1)$-arc-transitive.
\end{corollary}

To close this subsection, we give a short proof of the following result of Praeger~\cite[Theorem~3.1]{Praeger1989} using Lemma~\ref{lem18}.

\begin{proposition}\label{prop1}
Let $\Gamma$ be a $(G,2)$-arc-transitive digraph. If $G$ has a vertex-regular normal subgroup, then $\Gamma$ is a directed cycle.
\end{proposition}

\begin{proof}
Let $N$ be a vertex-regular normal subgroup of $G$, and $u\rightarrow v$ be an arc of $\Gamma$. Then $|G|/|N|=|G_v|$, and $G=G_{uv}N$ by Lemma~\ref{lem18}. Hence by Lemma~\ref{Factorization}(d), $|G_{uv}|\geqslant|G|/|N|=|G_v|$ and so $|G_{uv}|=|G_v|$. Consequently, $\Gamma$ is $1$-regular, which means that $\Gamma$ is a directed cycle.
\end{proof}

\subsection{Two technical lemmas}

\begin{lemma}\label{lem10}
Let $A$ be an almost simple group with socle $T$ and $L$ be a nonabelian simple group. Suppose $L^n\leqslant A$ and $|T|\leqslant|L^n|$ for some positive integer $n$. Then $n=1$ and $L=T$.
\end{lemma}

\begin{proof}
Note that $L^n/(L^n\cap T)\cong(L^nT)/T\leqslant A/T$, which is solvable by the Schreier conjecture. If $L^n\cap T\neq L^n$, then $L^n/(L^n\cap T)\cong L^m$ for some positive integer $m$, a contradiction. Hence $L^n\cap T=L^n$, which means $L^n\leqslant T$. This together with the condition that $|T|\leqslant|L^n|$ implies $L^n=T$. Hence $n=1$ and $L=T$, as the lemma asserts.
\end{proof}

\begin{lemma}\label{lem9}
Let $A$ be an almost simple group with socle $T$ and $S$ be a primitive permutation group on $|T|$ points. Then $S$ is not isomorphic to any subgroup of $A$.
\end{lemma}

\begin{proof}
Suppose for a contradiction that $S\lesssim A$. Regard $S$ as a subgroup of $A$, and write $\Soc(S)=L^n$ for some simple group $L$ and positive integer $n$. Since $S$ is primitive on $|T|$ points, $\Soc(S)$ is transitive on $|T|$ points, and so $|T|$ divides $|\Soc(S)|=|L|^n$. Consequently, $L$ is nonabelian for otherwise $T$ would be solvable. Then by Lemma~\ref{lem10} we have $\Soc(S)=L=T$. It follows that $S$ is an almost simple primitive permutation group with $\Soc(S)$ regular, contradicting~\cite{LPS1988}.
\end{proof}

\section{Vertex-quasiprimitive of type $\mathrm{SD}$}

\subsection{Constructing the graph $\Gamma(T)$}

\begin{construction}\label{Construction}
Let $T$ be a nonabelian simple group of order $k$ with $T=\{t_1,\dots,t_k\}$. Let $D=\{(t,\dots,t)\mid t\in T\}$ be a full diagonal subgroup of $T^k$ and let $g=(t_1,\dots,t_k)$. Define $\Gamma=\Cos(T^k,D,g)$ and let $V$ be the set of right cosets of $D$ in $T^k$, i.e. the vertex set of $\Gamma(T)$.
\end{construction}

\begin{lemma}\label{lem2}
$\Gamma(T)$ is a $|T|$-regular digraph.
\end{lemma}

\begin{proof}
Suppose that $D\cap g^{-1}Dg\neq1$. Then there exist $s,t\in T\setminus\{1\}$ such that $(s,\dots,s)=(t_1^{-1}tt_1,\dots,t_k^{-1}tt_k)$. Thus $s=t_i^{-1}tt_i$ for each $i$ such that $1\leqslant i\leqslant k$. Since $\{t_i\mid1\leqslant i\leqslant k\}=T$, we have $t_j=1$ for some $1\leqslant j\leqslant k$. It then follows from the equality $s=t_j^{-1}tt_j$ that $s=t$. Thus $t=t_i^{-1}tt_i$ for each $i$ such that $1\leqslant i\leqslant k$. Hence $t$ lies in the center of $T$, which implies $t=1$ as $T$ is nonabelian simple, a contradiction. Consequently, $D\cap g^{-1}Dg=1$, and so $\Cos(T^k,D,g)$ is $|T|$-regular as $|D|/|D\cap g^{-1}Dg|=|D|=|T|$.

Suppose that $g^{-1}\in DgD$. Then there exist $s,t\in T$ such that $(t_1^{-1},\dots,t_k^{-1})=(st_1t,\dots,st_kt)$. It follows that $t_i^{-1}=st_it$ for each $i$ such that $1\leqslant i\leqslant k$. Since $\{t_i\mid1\leqslant i\leqslant k\}=T$, we have $t_j=1$ for some $1\leqslant j\leqslant k$. Then the equality $t_j^{-1}=st_jt$ leads to $s=t^{-1}$. Thus $t_i^{-1}=t^{-1}t_it$ for each $i$ such that $1\leqslant i\leqslant k$. This implies that the inverse map is an automorphism of $T$ and so $T$ is abelian, a contradiction. Hence $g^{-1}\notin DgD$, from which we deduce that $\Cos(T^k,D,g)$ is a digraph, completing the proof.
\end{proof}

Next we show that up to isomorphism, the definition of $\Gamma(T)$ does not depend on the order of $t_1,t_2,\dots,t_k$.

\begin{lemma}\label{lem11}
Let $g'=(t'_1,\dots,t'_k)$ such that $T=\{t'_1,\dots,t'_k\}$. Then $\Cos(T^k,D,g)\cong\Cos(T^k,D,g')$.
\end{lemma}

\begin{proof}
Since $\{t_1',\dots,t'_k\}=\{t_1,\dots,t_k\}$, there exists $x\in\Sy_k$ such that $t_{i^x}=t'_i$ for each $i$ with $1\leqslant i\leqslant k$. Define an automorphism $\lambda$ of $T^k$ by $(g_1,\dots,g_k)^\lambda=(g_{1^x},\dots,g_{k^x})$ for all $(g_1,\dots,g_k)\in T^k$. Then $\lambda$ normalizes $D$ and $\lambda^{-1}g\lambda=g'$. Hence the map $Dh\mapsto Dh^\lambda$ gives an isomorphism from $\Cos(T^k,D,g)$ to $\Cos(T^k,D,g')$.
\end{proof}

For any $t\in T$, let $x(t)$ and $y(t)$ be the elements of $\Sy_k$ such that $t_{i^{x(t)}}=tt_i$ and $t_{i^{y(t)}}=t_it^{-1}$ for any $1\leqslant i\leqslant k$, and define permutations $\lambda(t)$ and $\rho(t)$ of $V$ by letting
$$
D(g_1,\dots,g_k)^{\lambda(t)}=D(g_{1^{x(t)}},\dots,g_{k^{x(t)}})\quad\text{and}\quad D(g_1,\dots,g_k)^{\rho(t)}=D(g_{1^{y(t)}},\dots,g_{k^{y(t)}})
$$
for any $(g_1,\dots,g_k)\in T^k$. For any $\varphi\in\Aut(T)$, let $z(\varphi)\in\Sy_k$ such that $t_{i^{z(\varphi)}}=t_i^\varphi$ for any $1\leqslant i\leqslant k$, and define $\delta(\varphi)\in\Sym(V)$ by letting
$$
D(g_1,\dots,g_k)^{\delta(\varphi)}=D((g_{1^{z(\varphi^{-1})}})^\varphi,\dots,(g_{k^{z(\varphi^{-1})}})^\varphi)
$$
for any $(g_1,\dots,g_k)\in T^k$. In particular, $\delta(\varphi)$ both permutes the coordinates and acts on each entry.

\begin{lemma}
$\lambda$ and $\rho$ are monomorphisms from $T$ to $\Sym(V)$, and $\delta$ is a monomorphism from $\Aut(T)$ to $\Sym(V)$.
\end{lemma}

\begin{proof}
For any $s,t\in T$, noting that $x(t)x(s)=x(st)$, we have
\begin{eqnarray*}
D(g_1,\dots,g_k)^{\lambda(s)\lambda(t)}&=&D(g_{1^{x(s)}},\dots,g_{k^{x(s)}})^{\lambda(t)}\\
&=&D(g_{1^{x(t)x(s)}},\dots,g_{k^{x(t)x(s)}})\\
&=&D(g_{1^{x(st)}},\dots,g_{k^{x(st)}})\\
&=&D(g_1,\dots,g_k)^{\lambda(st)}
\end{eqnarray*}
for each $(g_1,\dots,g_k)\in T^k$, and so $\lambda(st)=\lambda(s)\lambda(t)$. This means that $\lambda$ is a homomorphism from $T$ to $\Sym(V)$. Moreover, since $\lambda(t)$ acts on $V$ as the permutation $x(t)$ on the entries, $\lambda(t)=1$ if and only if $x(t)=1$, which is equivalent to $t=1$. Hence $\lambda$ is a monomorphism from $T$ to $\Sym(V)$. Similarly, $\rho$ is a monomorphism from $T$ to $\Sym(V)$.

For any $\varphi,\psi\in\Aut(T)$, since $z(\psi^{-1})z(\varphi^{-1})=z(\psi^{-1}\varphi^{-1})=z((\varphi\psi)^{-1})$, we have
\begin{eqnarray*}
D(g_1,\dots,g_k)^{\delta(\varphi)\delta(\psi)}&=&D((g_{1^{z(\varphi^{-1})}})^\varphi,\dots,(g_{k^{z(\varphi^{-1})}})^\varphi)^{\delta(\psi)}\\
&=&D((g_{1^{z(\psi^{-1})z(\varphi^{-1})}})^{\varphi\psi},\dots,(g_{k^{z(\psi^{-1})z(\varphi^{-1})}})^{\varphi\psi})\\
&=&D(g_1,\dots,g_k)^{\delta(\varphi\psi)}
\end{eqnarray*}
for all $(g_1,\dots,g_k)\in T^k$. This means that $\delta$ is a homomorphism from $\Aut(T)$ to $\Sym(V)$. Next we prove that $\delta$ is a monomorphism. Let $\varphi\in\Aut(T)$ such that
\begin{equation}\label{eq10}
D((g_{1^{z(\varphi^{-1})}})^\varphi,\dots,(g_{k^{z(\varphi^{-1})}})^\varphi)=D(g_1,\dots,g_k)^{\delta(\varphi)}=D(g_1,\dots,g_k)
\end{equation}
for each $(g_1,\dots,g_k)\in T^k$. Take any $i\in\{1,\dots,k\}$ and $(g_1,\dots,g_k)\in T^k$ such that $g_j=1$ for all $j\neq i$ and $g_i\neq1$. By~\eqref{eq10}, there exists $t\in T$ such that $(g_{j^{z(\varphi^{-1})}})^\varphi=tg_j$ for each $j\in\{1,\dots,k\}$. As a consequence, we obtain $t=1$ by taking any $j\in\{1,\dots,k\}\setminus\{i\}$ such that $j^{z(\varphi^{-1})}\neq i$. Also, for $j\in\{1,\dots,k\}$, $(g_{j^{z(\varphi^{-1})}})^\varphi\neq t$ if and only if $j=i$. It follows that $i^{z(\varphi^{-1})}=i$. As $i$ is arbitrary, this implies that $z(\varphi^{-1})=1$, and so $\varphi=1$. This shows that $\delta$ is a monomorphism from $\Aut(T)$ to $\Sym(V)$.
\end{proof}

Let $M$ be the permutation group on $V$ induced by the right multiplication action of $T^k$. For any group $X$, the \emph{holomorph} of $X$, denoted by $\Hol(X)$, is the normalizer of the right regular representation of $X$ in $\Sym(X)$. Note that $\langle x(T),y(T),z(\Aut(T))\rangle=x(T)\rtimes z(\Aut(T))=y(T)\rtimes z(\Aut(T))$ is primitive on $\{1,\dots,k\}$ and permutation isomorphic to $\Hol(T)$. Thus,
\begin{equation}\label{eq7}
X:=\langle M,\lambda(T),\rho(T),\delta(\Aut(T))\rangle
\end{equation}
is a primitive permutation group on $V$ of type $\mathrm{SD}$ with socle $M$, and the conjugation action of $X$ on the set of $k$ factors of $M\cong T^k$ is permutation isomorphic to $\Hol(T)$. Let $v=D\in V$, a vertex of $\Gamma(T)$. For any $t\in T$ let $\sigma(t)\in M$ be the permutation of $V$ induced by right multiplication by $(t,\dots,t)$. Then
$$
X_v/\sigma(T)=X_v/(X_v\cap M)\cong X_vM/M=X/M\cong\Hol(T)
$$
since $M$ acts transitively on $V$, and therefore
\begin{equation}\label{eq8}
|X_v|=|\sigma(T)||\Hol(T)|=|T|^3|\Out(T)|.
\end{equation}

\begin{lemma}\label{lem13}
$X\leqslant\Aut(\Gamma(T))$.
\end{lemma}

\begin{proof}
Clearly $M\leqslant\Aut(\Gamma(T))$, so it remains to verify that $\lambda(T)$, $\rho(T)$ and $\delta(\Aut(T))$ are subgroups of $\Aut(\Gamma(T))$. Let $t\in T$, $\varphi\in\Aut(T)$, $D(g_1,\dots,g_k)\in V$ and $D(g'_1,\dots,g'_k)\in V$. Then we have $D(g_1,\dots,g_k)\rightarrow D(g'_1,\dots,g'_k)$ in $\Gamma(T)$ if and only if
\begin{equation}\label{eq6}
(g'_1g_1^{-1},\dots,g'_kg_k^{-1})\in D(t_1,\dots,t_k)D.
\end{equation}
Since~\eqref{eq6} holds if and only if
\begin{eqnarray*}
(g'_{1^{x(t)}}g_{1^{x(t)}}^{-1},\dots,g'_{k^{x(t)}}g_{k^{x(t)}}^{-1})&\in&D(t_{1^{x(t)}},\dots,t_{k^{x(t)}})D\\
&=&D(tt_1,\dots,tt_k)D\\
&=&D(t_1,\dots,t_k)D,
\end{eqnarray*}
we conclude that $D(g_1,\dots,g_k)\rightarrow D(g'_1,\dots,g'_k)$ if and only if $D(g_1,\dots,g_k)^{\lambda(t)}\rightarrow D(g'_1,\dots,g'_k)^{\lambda(t)}$. This shows $\lambda(t)\in\Aut(\Gamma(T))$ for any $t\in T$. Similarly, we have $\rho(t)\in\Aut(\Gamma(T))$ for any $t\in T$. Also, \eqref{eq6} holds if and only if
\begin{eqnarray*}
((g'_{1^{z(\varphi^{-1})}}g_{1^{z(\varphi^{-1})}}^{-1})^\varphi,\dots,(g'_{k^{z(\varphi^{-1})}}g_{k^{z(\varphi^{-1})}}^{-1})^\varphi)
&\in&D((t_{1^{z(\varphi^{-1})}})^\varphi,\dots,(t_{k^{z(\varphi^{-1})}})^\varphi)D\\
&=&D((t_1^{\varphi^{-1}})^\varphi,\dots,(t_k^{\varphi^{-1}})^\varphi)D\\
&=&D(t_1,\dots,t_k)D.
\end{eqnarray*}
It follows that $D(g_1,\dots,g_k)\rightarrow D(g'_1,\dots,g'_k)$ if and only if $D(g_1,\dots,g_k)^{\delta(\varphi)}\rightarrow D(g'_1,\dots,g'_k)^{\delta(\varphi)}$, and so $\delta(\varphi)\in\Aut(\Gamma(T))$ for any $\varphi\in\Aut(T)$. This completes the proof.
\end{proof}

Denote $H=\langle M,\lambda(T)\rangle=M\rtimes\lambda(T)\leqslant X$.

\begin{lemma}\label{lem5}
$\Gamma(T)$ is $(H,2)$-arc-transitive.
\end{lemma}

\begin{proof}
It is readily seen that $H_v=\sigma(T)\times\lambda(T)\cong T^2$. Let $K=\{\sigma(t)\lambda(t)\mid t\in T\}$. For any $t\in T$ and any $(g_1,\dots,g_k)\in T^k$ we have
\begin{eqnarray*}
D(g_1,\dots,g_k)^{g^{-1}\sigma(t)\lambda(t)g}&=&D(g_1t_1^{-1}t,\dots,g_kt_k^{-1}t)^{\lambda(t)g}\\
&=&D(g_{1^{x(t)}}t_{1^{x(t)}}^{-1}t,\dots,g_{k^{x(t)}}t_{k^{x(t)}}^{-1}t)^g\\
&=&D(g_{1^{x(t)}}(tt_1)^{-1}t,\dots,g_{k^{x(t)}}(tt_k)^{-1}t)^g\\
&=&D(g_{1^{x(t)}}t_1^{-1},\dots,g_{k^{x(t)}}t_k^{-1})^g\\
&=&D(g_{1^{x(t)}},\dots,g_{k^{x(t)}})\\
&=&D(g_1,\dots,g_k)^{\lambda(t)}.
\end{eqnarray*}
Hence $g^{-1}\sigma(t)\lambda(t)g=\lambda(t)$ for all $t\in T$. Consequently, $g^{-1}Kg=\lambda(T)<H_v$ and so $K\leqslant H_v\cap gH_vg^{-1}$. Now for any elements $s$ and $t$ of $T$,
$$
\sigma(s)\lambda(t)=(\sigma(s)\lambda(s))\lambda(s^{-1}t)\in K\lambda(T)=K(g^{-1}Kg).
$$
It follows that
$$
H_v\leqslant K(g^{-1}Kg)\leqslant(H_v\cap gH_vg^{-1})(H_v\cap g^{-1}H_vg),
$$
so $H_v=(H_v\cap gH_vg^{-1})(H_v\cap g^{-1}H_vg)$. Thus by Remark~\ref{rem1}, $\Gamma(T)$ is $(H,2)$-arc-transitive, as the lemma asserts.
\end{proof}

An immediate consequence of Lemma~\ref{lem5} is that $\Gamma(T)$ is $(X,2)$-arc-transitive. However, $X$ is not transitive on the set of $3$-arcs of $\Gamma(T)$, as we shall see in the next lemma.

\begin{lemma}\label{lem14}
$\Gamma(T)$ is not $(X,3)$-arc-transitive.
\end{lemma}

\begin{proof}
Suppose that $\Gamma(T)$ is $(X,3)$-arc-transitive. Then since $M$ is a vertex-transitive normal subgroup of $X$, Corollary~\ref{lem17} asserts that $\Gamma(T)$ is $(M,2)$-arc-transitive. As a consequence, $M_v$ is transitive on $A_2(v):=\{(v_1,v_2)\in V^2\mid v\rightarrow v_1\rightarrow v_2\}$, the set of $2$-arcs starting from $v$. However, $|M_v|=|T|$ while $|A_2(v)|=|T|^2$ as $\Gamma(T)$ is $|T|$-regular. This is not possible.
\end{proof}

\subsection{Classification}

Throughout this subsection, let $T$ be a nonabelian simple group, $k\geqslant2$ be an interger, $D=\{(t,\dots,t)\mid t\in T\}$ be a full diagonal subgroup of $T^k$, $V$ be the set of right cosets of $D$ in $T^k$, and $M$ be the permutation group induced by the right multiplication action of $T^k$ on $V$. Suppose that $G$ is a permutation group on $V$ with $M\leqslant G\leqslant M.(\Out(T)\times\Sy_k)$, and $\Gamma$ is a connected $(G,2)$-arc-transitive digraph. Let $v=D\in V$ and $w$ be an out-neighbour of $v$. Then $w=D(t_1,\dots,t_k)\in V$ for some elements $t_1,\dots,t_k$ of $T$ which are not all equal. Without loss of generality, we assume $t_k=1$. Let $u=D(t_1^{-1},\dots,t_k^{-1})\in V$ and $g\in M$ be the permutation of $V$ induced by right multiplication by $(t_1,\dots,t_k)\in T^k$. Moreover, define $\{\Omega_1,\dots,\Omega_n\}$ to be the partition of $\{1,\dots,k\}$ such that $t_i=t_j$ if and only if $i$ and $j$ are in the same part of $\{\Omega_1,\dots,\Omega_n\}$. Note that $G_v\leqslant\Aut(T)\times\Sy_k$. Let $\alpha$ be the projection of $G_v$ into $\Aut(T)$ and $\beta$ be the projection of $G_v$ into $\Sy_k$. Let $A=\alpha(G_v)$ and $S=\beta(G_v)$, so that $G_v\leqslant A\times S$, where each element $\sigma$ of $A$ is induced by an automorphism of $T$ acting on $V$ as
$$
D(g_1,\dots,g_k)^\sigma=D(g_1^\sigma,\dots,g_k^\sigma)
$$
and each element $x$ of $S$ is induced by a permutation on $\{1,\dots,k\}$ acting on $V$ as
$$
D(g_1,\dots,g_k)^x=D(g_{1^{x^{-1}}},\dots,g_{k^{x^{-1}}}).
$$
As $G\geqslant M$ we have $\Inn(T)\leqslant A\leqslant\Aut(T)$. Moreover, since $G$ is $2$-arc-transitive, Lemma~\ref{lem19} implies that $G_v=G_{uv}G_{vw}$. Let $R$ be the stabilizer in $S$ of $k$ in the set $\{1,\dots,k\}$.

Take any $\sigma\in A$ and $x\in S$. Then $\sigma x\in G_u$ if and only if $x^{-1}\sigma^{-1}$ fixes $u$, that is
$$
D((t_{1^x}^{-1})^{\sigma^{-1}},\dots,(t_{(k-1)^x}^{-1})^{\sigma^{-1}},(t_{k^x}^{-1})^{\sigma^{-1}})=D(t_1^{-1},\dots,t_{k-1}^{-1},1),
$$
or equivalently,
\begin{equation}\label{eq1}
D(t_{k^x}t_{1^x}^{-1},\dots,t_{k^x}t_{(k-1)^x}^{-1},1)=D((t_1^{-1})^\sigma,\dots,(t_{k-1}^{-1})^\sigma,1).
\end{equation}
Similarly, $\sigma x\in G_w$ if and only if $x^{-1}\sigma^{-1}$ fixes $w$, which is equivalent to
\begin{equation}\label{eq2}
D(t_{k^x}^{-1}t_{1^x},\dots,t_{k^x}^{-1}t_{(k-1)^x},1)=D(t_1^\sigma,\dots,t_{k-1}^\sigma,1).
\end{equation}

\begin{lemma}\label{lem6}
$\langle t_1,\dots,t_k\rangle=T$.
\end{lemma}

\begin{proof}
For all $\sigma\in\alpha(G_{uv})$, there exists $x\in S$ such that $\sigma x\in G_u$. Then~\eqref{eq1} implies that $t_{k^x}t_{i^x}^{-1}=(t_i^{-1})^\sigma$ and thus $t_i^\sigma=t_{i^x}t_{k^x}^{-1}$ for all $i$ such that $1\leqslant i\leqslant k$. This shows that $\alpha(G_{uv})$ stabilizes $\langle t_1,\dots,t_k\rangle$. Similarly, for all $\sigma\in\alpha(G_{vw})$, there exists $x\in S$ such that $\sigma x\in G_w$. Then~\eqref{eq2} implies that $t_i^\sigma=t_{k^x}^{-1}t_{i^x}$ for all $i$ such that $1\leqslant i\leqslant k$. Accordingly, $\alpha(G_{vw})$ also stabilizes $\langle t_1,\dots,t_k\rangle$. It follows that $A=\alpha(G_v)=\alpha(G_{uv}G_{vw})=\alpha(G_{uv})\alpha(G_{vw})$ stabilizes $\langle t_1,\dots,t_k\rangle$. Hence $\langle t_1,\dots,t_k\rangle=T$ since $\Inn(T)\leqslant A\leqslant\Aut(T)$.
\end{proof}

\begin{lemma}\label{lem3}
$G_{uv}\cap(A\times R)=G_{vw}\cap(A\times R)$.
\end{lemma}

\begin{proof}
Let $\sigma\in A$ and $x\in R$. Then $t_{k^x}=t_k=1$, and thus~\eqref{eq2} shows that $\sigma x\in G_w$ if and only if $t_{i^x}=t_i^\sigma$ for all $i$ such that $1\leqslant i\leqslant k$. Similarly, \eqref{eq1} shows that $\sigma x\in G_u$ if and only if $t_{i^x}^{-1}=(t_i^{-1})^\sigma$ for all $i$ such that $1\leqslant i\leqslant k$. Since this is equivalent to $t_{i^x}=t_i^\sigma$ for all $i$, we conclude that $\sigma x\in G_w$ if and only if $\sigma x\in G_u$. As a consequence, $G_{uv}\cap(A\times R)=G_{vw}\cap(A\times R)$.
\end{proof}

\begin{lemma}\label{lem7}
$G_{uv}\cap A=G_{vw}\cap A=1$.
\end{lemma}

\begin{proof}
In view of Lemma~\ref{lem3} we only need to prove that $G_{vw}\cap A=1$. For any $\sigma\in G_{vw}\cap A$, \eqref{eq2} shows that $D(t_1,\dots,t_{k-1},1)=D(t_1^\sigma,\dots,t_{k-1}^\sigma,1)$, and so $t_i^\sigma=t_i$ for all $i$ such that $1\leqslant i\leqslant k$. By Lemma~\ref{lem6}, this implies that $\sigma=1$ and so $G_{vw}\cap A=1$, as desired.
\end{proof}

\begin{lemma}\label{lem1}
Both $\beta(G_{uv})$ and $\beta(G_{vw})$ preserve the partition $\{\Omega_1,\dots,\Omega_n\}$.
\end{lemma}

\begin{proof}
Let $x\in\beta(G_{uv})$. Then there exists $\sigma\in A$ such that $\sigma x\in G_u$, and so~\eqref{eq1} gives
\begin{equation}\label{eq3}
t_{k^x}t_{i^x}^{-1}=(t_i^{-1})^\sigma
\end{equation}
for all $i$ such that $1\leqslant i\leqslant k$. For any $i,j\in\{1,\dots,k\}$, if $i$ and $j$ are in the same part of $\{\Omega_1,\dots,\Omega_n\}$, then $t_i=t_j$ and so $(t_i^{-1})^\sigma=(t_j^{-1})^\sigma$, which leads to $t_{i^x}=t_{j^x}$ by~\eqref{eq3}. Since $t_{i^x}=t_{j^x}$ if and only if $i^x$ and $j^x$ are in the same part of $\{\Omega_1,\dots,\Omega_n\}$, this shows that $x$, hence $\beta(G_{uv})$, preserves the partition $\{\Omega_1,\dots,\Omega_n\}$. The proof for $\beta(G_{vw})$ is similar.
\end{proof}

\begin{lemma}\label{lem4}
$t_1,\dots,t_k$ are pairwise distinct.
\end{lemma}

\begin{proof}
Let $U$ be the subset of $V$ consisting of the elements $D(g_1,\dots,g_k)$ with $g_i=g_j$ whenever $i$ and $j$ are in the same part of $\{\Omega_1,\dots,\Omega_n\}$. By Lemma~\ref{lem1}, both $\beta(G_{uv})$ and $\beta(G_{vw})$ preserve the partition $\{\Omega_1,\dots,\Omega_n\}$. Then since $S=\beta(G_v)=\beta(G_{uv}G_{vw})=\beta(G_{uv})\beta(G_{vw})$, we derive that $S$ preserves the partition $\{\Omega_1,\dots,\Omega_n\}$. As a consequence, $S$ stabilizes $U$ setwise. Meanwhile, $A$ and $g$ stabilize $U$ setwise. Hence $G=\langle G_v,g\rangle\leqslant\langle A\times S,g\rangle$ stabilizes $U$ setwise, which implies $U=V$. Thus each $\Omega_i$ has size $1$ and so $t_1,\dots,t_k$ are pairwise distinct.
\end{proof}

\begin{lemma}\label{lem8}
$G_{uv}\cap R=G_{vw}\cap R=1$.
\end{lemma}

\begin{proof}
In view of Lemma~\ref{lem3} we only need to prove that $G_{vw}\cap R=1$. Let $x\in G_{vw}\cap R$. Then $t_{k^x}=t_k=1$, and so~\eqref{eq2} shows that $t_{i^x}=t_i$ for all $i$ such that $1\leqslant i\leqslant k$. Note that $t_1,\dots,t_k$ are pairwise distinct by Lemma~\ref{lem4}. We conclude that $x=1$ and so $G_{vw}\cap R=1$, as desired.
\end{proof}

\begin{lemma}\label{lem12}
$k=|T|$, $\{t_1,\dots,t_k\}=T$ and $\Gamma\cong\Gamma(T)$ as given in \emph{Construction~\ref{Construction}}. Moreover, if $G$ is vertex-primitive, then the induced permutation group of $G$ on the $k$ copies of $T$ is a subgroup of $\Hol(T)$ containing $\Soc(\Hol(T))$.
\end{lemma}

\begin{proof}
It follows from Lemma~\ref{lem3} that $G_{uvw}\cap(A\times R)=G_{uv}\cap(A\times R)$. Then as $G$ is $2$-arc-transitive on $\Gamma$, we have
\begin{eqnarray}\label{eq5}
\frac{|G_v|}{|G_{uv}|}=\frac{|G_{uv}|}{|G_{uvw}|}&\leqslant&\frac{|G_{uv}|}{|G_{uvw}\cap(A\times R)|}\\
&=&\frac{|G_{uv}|}{|G_{uv}\cap(A\times R)|}=\frac{|G_{uv}(A\times R)|}{|A\times R|}\leqslant\frac{|A\times S|}{|A\times R|}=k.\nonumber
\end{eqnarray}
We thus obtain $|G_v|\leqslant k|G_{uv}|=k|G_{vw}|$. From Lemma~\ref{lem7} we deduce $\beta(G_{uv})\cong G_{uv}$ and $\beta(G_{vw})\cong G_{vw}$. Moreover, $t_1,\dots,t_k$ are pairwise distinct by Lemma~\ref{lem4}, which implies $|T|\geqslant k$. Therefore,
$$
k|S|\leqslant|T||S|\leqslant|G_v\cap A||S|=|G_v|\leqslant k|G_{uv}|=k|\beta(G_{uv})|\leqslant k|S|
$$
and
$$
k|S|\leqslant|T||S|\leqslant|G_v\cap A||S|=|G_v|\leqslant k|G_{vw}|=k|\beta(G_{vw})|\leqslant k|S|.
$$
Hence $|G_v\cap A|=|T|=k$, $|G_v|=k|G_{uv}|=k|G_{vw}|$ and $\beta(G_{uv})=\beta(G_{vw})=S$. As a consequence, $T=\{t_1,\dots,t_k\}$ by Lemma~\ref{lem4}, and so $\Gamma\cong\Cos(T^k,D,g)\cong\Gamma(T)$. Also, \eqref{eq5} implies that $G_{uvw}=G_{uvw}\cap(A\times R)$. If $G_{uv}\cap S=1$ or $G_{vw}\cap S=1$, then Lemma~\ref{lem7} implies $S=\beta(G_{uv})\cong G_{uv}\lesssim A$ or $S=\beta(G_{vw})\cong G_{vw}\lesssim A$, contradicting Lemma~\ref{lem9}. Thus $G_{uv}\cap S$ and $G_{vw}\cap S$ are both nontrivial normal subgroups of $\beta(G_{uv})=\beta(G_{vw})=S$.

From now on suppose that $G$ is primitive and so $S$ is a primitive subgroup of $\Sy_k$. By Lemma~\ref{lem8}, $G_{uv}\cap R=G_{vw}\cap R=1$, so we derive that $G_{uv}\cap S$ and $G_{vw}\cap S$ are both regular normal subgroups of $S$. Moreover, $G_{uv}\cap S\neq G_{vw}\cap S$ for otherwise $G_{uvw}\cap S=G_{uv}\cap S$ would be a regular subgroup of $S$, contrary to the condition $G_{uvw}=G_{uvw}\cap(A\times R)\leqslant A\times R$. This indicates that $S$ has at least two regular normal subgroups, and so $\Soc(S)=N^{2n}$ for some nonabelian simple group $N$ and positive integer $n$ such that $k=|N|^n$ and $S/(G_{uv}\cap S)$ has a normal subgroup isomorphic to $N^n$. It follows that
$$
N^n\lesssim S/(G_{uv}\cap S)=\beta(G_{uv})/(G_{uv}\cap S)\cong\alpha(G_{uv})/(G_{uv}\cap A)\cong\alpha(G_{uv})\leqslant A,
$$
and then Lemma~\ref{lem10} implies that $n=1$ and $N\cong T$. Thus, $\Soc(S)\cong T^2$, and so $\Soc(\Hol(T))\leqslant S\leqslant\Hol(T)$.
\end{proof}

We are now ready to give the main theorem of this section. Recall $X$ defined in~\eqref{eq7}.

\begin{theorem}\label{thm3}
Let $T$ be a nonabelian simple group, $k\geqslant2$ be an interger, and $T^k\leqslant G\leqslant T^k.(\Out(T)\times\Sy_k)$ with diagonal action on the set $V$ of right cosets of $\{(t,\dots,t)\mid t\in T\}$ in $T^k$. Suppose $\Gamma$ is a connected $(G,2)$-arc-transitive digraph with vertex set $V$. Then $k=|T|$, $\Gamma\cong\Gamma(T)$, $\Aut(\Gamma)=X$ is vertex-primitive of type SD with socle $T^k$ and the conjugation action on the $k$ copies of $T$ permutation isomorphic to $\Hol(T)$, and $\Gamma$ is not $3$-arc-transitive.
\end{theorem}

\begin{proof}
We have by Lemma~\ref{lem12} that $k=|T|$, $\{t_1,\dots,t_k\}=T$ and $\Gamma\cong\Gamma(T)$. In the following, we identify $\Gamma$ with $\Gamma(T)$. Let $X$ be as in~\eqref{eq7} and $Y=\Aut(\Gamma(T))$. Then $X$ is vertex-primitive of type $\mathrm{SD}$ with socle $T^{|T|}$, and the conjugation action of $X$ on the $|T|$ copies of $T$ is permutation isomorphic to $\Hol(T)$. Also, $X\leqslant Y$ by Lemma~\ref{lem13}. It follows from~\cite[Theorem~1.2]{BP2003} that $Y$ is vertex-primitive of type $\mathrm{SD}$ with the same socle of $X$. Then again by Lemma~\ref{lem12} we have $Y_v\leqslant\Aut(T)\times\Hol(T)$. Thus by~\eqref{eq8} $Y_v=X_x$. Since $X$ is vertex-transitive, it follows that $Y=XY_v=X$, and so $\Gamma$ is not $3$-arc-transitive by Lemma~\ref{lem14}.
\end{proof}

\section{Product action on the vertex set}

In this section, we study $(G,s)$-arc-transitive digraphs with vertex set $\Delta^m$ such that $G$ acts on $\Delta^m$ by product action. We first prove Theorem~\ref{thm2}.

\vskip0.1in
\noindent\emph{Proof of Theorem~\ref{thm2}.}
Let $G_1$ be the stabiliser in $G$ of the first coordinate and $\pi_1$ be the projection of $G_1$ into $\Sym(\Delta)$. Then $\pi_1(G_1)=H$. Since $N$ is normal in $H$ and transitive on $\Delta$, $N^m$ is normal in $G$ and transitive on $\Delta^m=V$. Hence Corollary~\ref{lem17} implies that $\Gamma$ is $(N^m,s-1)$-arc-transitive. In particular, since $s\geqslant2$, $N^m$ is transitive on the set of arcs of $\Gamma$, and so $\Gamma$ has arc set $A=\{u^n\rightarrow v^n\mid n\in N^m\}$ for any arc $u\rightarrow v$ of $\Gamma$.

Let $\alpha\in\Delta$, $u=(\alpha,\dots,\alpha)\in V$ and $v=(\beta_1,\dots,\beta_m)$ be an out-neighbour of $u$ in $\Gamma$. By Lemma~\ref{lem18} we have $G=N^mG_{uv}$. Let $\varphi$ be the projection of $G$ to $\Sy_m$, and we regard $\varphi(G)$ as a subgroup of $\Sym(V)$.  Then
$$
\varphi(G)\leqslant H^mG=H^m(N^mG_{uv})=H^mG_{uv}.
$$
Take any $i$ in $\{1,\dots,m\}$. Since $\varphi(G)$ is transitive on $\{1,\dots,m\}$, there exists $x\in\varphi(G)$ such that $1^x=i$ and  $x=yz$ with $y=(y_1,\dots,y_m)\in H^m$ and $z\in G_{uv}$. Note that $z\in G_{uv}$ and $x\in\Sy_m$ both fix $u$. We conclude that $y$ fixes $u$ and hence $y_j\in H_\alpha$ for each $j$ in $\{1,\dots,m\}$. Also, $y^{-1}x=z\in G_{uv}\leqslant G_v$ implies $\beta_1^{y_1^{-1}}=\beta_i$. It follows that for each $i$ in $\{1,\dots,m\}$ there exists $h_i\in H_\alpha$ with $\beta_i^{h_i}=\beta_1$. Let $w=(\beta_1,\dots,\beta_1)\in V$, $h=(h_1,\dots,h_m)\in(H_\alpha)^m$ and $\Gamma^h$ be the digraph with vertex set $V$ and arc set $A^h:=\{u^{nh}\rightarrow v^{nh}\mid n\in N^m\}$. It is evident that $u^h=u$, $v^h=w$, and $h$ gives an isomorphism from $\Gamma$ to $\Gamma^h$. Let $\Sigma$ be the digraph with vertex set $\Delta$ and arc set $I:=\{\alpha^n\rightarrow\beta_1^n\mid n\in N\}$. Then $N\leqslant\Aut(\Sigma)$, and viewing $N^mh=hN^m$ we have
\begin{eqnarray*}
A^h&=&\{u^{hn}\rightarrow v^{hn}\mid n\in N^m\}=\{u^n\rightarrow w^n\mid n\in N^m\}\\
&=&\{(\alpha^{n_1},\dots,\alpha^{n_m})\rightarrow(\beta_1^{n_1},\dots,\beta_1^{n_m})\mid n_1,\dots,n_m\in N\}.
\end{eqnarray*}
This implies that $\Gamma^h=\Sigma^m$. Consequently, $\Gamma\cong\Sigma^m$.

For any $\beta\in\Delta$, denote by $\delta(\beta)$ the point in $V=\Delta^m$ with all coordinates equal to $\beta$. Then $\delta(\alpha)\rightarrow\delta(\beta_1)$ in $\Gamma^h$ since $\alpha\rightarrow\beta_1$ in $\Sigma$. Let $x$ be any element of $H$. Then since
\[
H=h_1^{-1}Hh_1=h_1^{-1}\pi_1(G_1)h_1=\pi_1(h)^{-1}\pi_1(G_1)\pi_1(h)=\pi_1(h^{-1}G_1h),
\]
there exists $g\in h^{-1}G_1h$ such that $x=\pi_1(g)$. As $g$ is an automorphism of $\Gamma^h$ and $\delta(\alpha)\rightarrow\delta(\beta_1)$ in $\Gamma^h$, we have $\delta(\alpha)^g\rightarrow\delta(\beta_1)^g$ in $\Gamma^h$. Comparing first coordinates, this implies that $\alpha^{\pi_1(g)}\rightarrow\beta_1^{\pi_1(g)}$ in $\Sigma$, which turns out to be $\alpha^x\rightarrow\beta_1^x$ in $\Sigma$. In other words, $\alpha^x\rightarrow\beta_1^x$ is in $I$. It follows that
\[
I=\{\alpha^{xn}\rightarrow\beta_1^{xn}\mid n\in N\}=\{\alpha^{nx}\rightarrow\beta_1^{nx}\mid n\in N\}
\]
as $xN=Nx$. Hence $x$ preserves $I$, and so $H\leqslant\Aut(\Sigma)$.

Let $\alpha_0\rightarrow\alpha_1\rightarrow\dots\rightarrow\alpha_s$ be an $s$-arc of $\Sigma$. Since $\Gamma$ is $(N^m,s-1)$-arc-transitive and $N^m=h^{-1}N^mh$, it follows that $\Gamma^h$ is $(N^m,s-1)$-arc-transitive. Then for any $(s-1)$-arc $\alpha'_1\rightarrow\dots\rightarrow\alpha'_s$ of $\Sigma$, since $\delta(\alpha_1)\rightarrow\dots\rightarrow\delta(\alpha_s)$ and $\delta(\alpha'_1)\rightarrow\dots\rightarrow\delta(\alpha'_s)$ are both $(s-1)$-arcs of $\Gamma^h$, there exists $(n_1,\dots,n_m)\in N^m$ such that $\delta(\alpha_i)^{(n_1,\dots,n_m)}=\delta(\alpha'_i)$ for each $i$ with $1\leqslant i\leqslant s$. Hence $\alpha_i^{n_1}=\alpha'_i$ for each $i$ with $1\leqslant i\leqslant s$. Therefore, $\Sigma$ is $(N,s-1)$-arc-transitive. Let $\Sigma^+(\alpha_{s-1})$ be the set of out-neighbours of $\alpha_{s-1}$ in $\Sigma$. Take any $\beta\in\Sigma^+(\alpha_{s-1})$. As $\delta(\alpha_s)$ and $\delta(\beta)$ are both out-neighbours of $\delta(\alpha_{s-1})$ in $\Gamma^h$ and $\Gamma^h$ is $(h^{-1}Gh,s)$-arc-transitive, there exists $g\in h^{-1}Gh\leqslant H\wr\Sy_m$ such that $g$ fixes $\delta(\alpha_0),\delta(\alpha_1),\dots,\delta(\alpha_{s-1})$ and maps $\delta(\alpha_s)$ to $\delta(\beta)$. Write $g=(x_1,\dots,x_m)\sigma$ with $(x_1,\dots,x_m)\in H^m$ and $\sigma\in\Sy_m$. Then $x_1$ fixes $\alpha_0,\alpha_1,\dots,\alpha_{s-1}$ and maps $\alpha_s$ to $\beta$. This shows that $\Sigma$ is $(H,s)$-arc-transitive, completing the proof.
\qed

\vskip0.1in
\noindent\textsc{Acknowledgements.} This research was supported by Australian Research Council grant DP150101066. The authors would like to thank the anonymous referee for helpful comments.

\end{document}